\documentclass[12pt,twoside,reqno]{amsart}

\usepackage[OT1]{fontenc}    %
\usepackage{type1cm}         %
\usepackage[dvips]{graphicx} % for gfx
\usepackage[active]{srcltx}  % for the inverse search
\usepackage{verbatim}        % comment the following line to get rid
\numberwithin{equation}{section}

\theoremstyle{plain}
\newtheorem{theorem}{Theorem}[section]
\newtheorem{corollary}[theorem]{Corollary}
\newtheorem{prop}[theorem]{Proposition}
\newtheorem{lemma}[theorem]{Lemma}

\theoremstyle{remark}
\newtheorem{remark}[theorem]{Remark}

\newtheorem{question}[theorem]{Question}

\theoremstyle{definition}

\newcommand{\supp}{\textrm{Supp}}

\newcommand{\e}{\varepsilon}

\newcommand{\dimp}{\dim_P}

\newcommand{\R}{\mathbb{R}}

\newcommand{\DD}{\mathcal{D}}
\newcommand{\RR}{\mathcal{R}}

\newcommand{\EE}{\mathbf{E}}

\newcommand{\wt}{\widetilde}

\DeclareMathOperator{\por}{por}
\DeclareMathOperator{\esssup}{ess\,sup}

\title[The dimension of mean porous measures]{The dimension of weakly mean porous measures: a probabilistic approach}

\author{Pablo Shmerkin}
\address{Pablo Shmerkin\\
         University of Manchester. Oxford Road\\
         School of Mathematics. Alan Turing Building\\
         Manchester M13 9PL\\
         UK }

\email{Pablo.Shmerkin@manchester.ac.uk}

\thanks{I acknowledge support from EPSRC grant EP/E050441/1 and the University of Manchester.}


\date{\today}

\begin{document}

\begin{abstract}
Using probabilistic ideas, we prove that if $\mu$ is a mean porous measure on $\R^n$, then the packing dimension of $\mu$ is strictly smaller than $n$. Moreover, we give an explicit bound for the packing dimension, which is asymptotically sharp in the case of small porosity. This result was stated in [D. B. BELIAEV and S. K. SMIRNOV, ``On dimension of porous measures'', \textit{Math. Ann.} 323 (2002) 123-141], but the proof given there is not correct. We also give estimates on the dimension of weakly mean porous measures, which improve another result of Beliaev and Smirnov.
\end{abstract}

\maketitle

\section{Introduction and statement of results}

\subsection{Background and definitions} Porosity and dimension are two useful but different concepts for quantifying the degree of singularity of fractal sets and measures; thus it is a natural problem to understand the relationship between them. This problem has received considerable attention over the last two decades, and continues to be an active area of research. Estimates for dimension in terms of porosity were obtained for a wide variety of notions of porosity (and dimension) \cite{MartioVuorinen87, Mattila88, Salli91, KoskelaRohde97, EJJ00, BeliaevSmirnov02, JarvenpaaJarvenpaa02, JJKS05, Nieminen06, Kaenmaki07, KaenmakiSuomala08, BJJKRSS09,  JJKRRS09, KaenmakiSuomala09, KRS09, Rajala09}, and the porosity (in an appropriate sense) of many natural sets and measures was investigated  \cite{KoskelaRohde97, Urbanski01, BeliaevSmirnov02, JJM02, Chousionis09}. Moreover, the relationship between porosity and other geometric concepts such as conical densities and singular integrals was explored \cite{Mattila88, KaenmakiSuomala08, KaenmakiSuomala09, Chousionis09}.

 This article is a contribution to this line of research. Our results yield a strong quantitative form of a theorem that was stated in \cite{BeliaevSmirnov02}, but with an incorrect proof. The methods we employ are rather different from those used in previous research on porosities and dimension.

We start by reviewing some of the main definitions and results. Let $E$ be a subset of $\R^n$. Given $x\in E$ and $r>0$, we define
\[
\por(E,x,r) = \sup\{\alpha\ge 0: \exists y \textrm{ such that } B(y,\alpha r)\subset B(x,r)\setminus E \}.
\]
Thus, $\por(E,x,r)$ denotes the relative size of the largest hole in $E$ one can find around $x$ at scale $r$. The \textbf{porosity} of $E$ is then defined as
\[
\por(E) = \sup\left\{ \alpha\ge 0: \liminf_{r\rightarrow 0} \por(E,x,r) \ge \alpha \text{ for all } x\in E\right\}.
\]
It is rather easy to see that $0\le \por(E)\le 1/2$, and it seems intuitively clear that if $\por(E)>0$ then $E$ cannot have full dimension. This is true, and remains valid under the weaker assumption of mean porosity, introduced by Koskela and Rohde in \cite{KoskelaRohde97}: we say that $E\subset\mathbb{R}^n$ is \textbf{mean $(\alpha,\eta)$-porous} if, for every $x\in E$,
\[
\liminf_{n\rightarrow \infty} \frac{1}{n}|\{ i\in[n]: \por(E,x,2^{-i}) \ge\alpha   \}| \ge \eta,
\]
where $[n]=\{0,\ldots, n-1\}$. It is shown in \cite[Theorem 2.1]{KoskelaRohde97} that if $E$ is mean $(\alpha,\eta)$-porous, then
\begin{equation} \label{eq:dimension-mean-porous-sets}
\dim_P(E) \le d - c_d \cdot \eta \cdot \alpha^d,
\end{equation}
where $\dim_P$ denotes packing dimension (see e.g. \cite[\S 5.9]{Mattila95} for its definition), and $c_d$ depends only on the ambient dimension.\footnote{To be precise, in \cite{KoskelaRohde97} this is shown for upper Minkowski dimension under a uniform version of mean porosity; from here one deduces \eqref{eq:dimension-mean-porous-sets} by standard arguments.}  Koskela and Rohde also show that, other than for the value of the constant, this estimate is asymptotically sharp as $\alpha\to 0$.

In this article, we use probabilistic and dynamical ideas to obtain an analogue of the above result for measures rather than sets. The notion of porosity was extended to measures by Eckman, J\"{a}rvenp\"{a}\"{a} and J\"{a}rvenp\"{a}\"{a} in \cite{EJJ00}, and mean porosity of measures was introduced by Beliaev and Smirnov in \cite{BeliaevSmirnov02}.

Let $\mu$ be a Borel probability measure on $\R^d$. Given $x\in\supp\mu$, $r>0$ and $\e>0$, we let
\begin{equation} \label{eq:def-porosity}
\por(\mu,x,r,\e) = \sup\left\{ \alpha: \inf_{ B(y,\alpha r)\subset B(x,r) } \frac{\mu(B(y,\alpha r))}{\mu(B(x,r))} \le \e    \right\}.
\end{equation}
Thus, $\por(\mu,x,r,\e)$ is the largest (relative) size of an ``$\e$-hole'' in $B(x,r)$. The porosity of $\mu$ at a point $x$ is then defined as
\[
\por(\mu,x) =  \lim_{\e\to 0}\liminf_{r\to 0} \por(\mu,x,r,\e).
\]
Finally, the \textbf{porosity} of $\mu$ is the $\mu$-essential infimum of $\por(\mu,x)$.\footnote{Sometimes one takes the essential supremum instead; all our results continue to hold in this case, after replacing packing dimension by lower packing dimension.}

In this article we will be concerned with the more general notion of mean porosity. If $0<\alpha\le 1/2$, $0<\eta\le 1$ and $0\le \e<1$, we say that $\mu$ is \textbf{mean $(\alpha,\eta,\e)$-porous} if
\[
\liminf_{n\rightarrow \infty} \frac{1}{n}|\{i\in [n]:\por(\mu,x,2^{-i},\e)\ge \alpha\}| \ge \eta,
\]
for $\mu$-almost every $x$.

Finally, we say that $\mu$ is \textbf{mean $(\alpha,\eta)$-porous} if $\mu$ is mean $(\alpha,\eta,\e)$-porous for all $\e>0$.\footnote{We have followed the definition from \cite{BJJKRSS09}. Beliaev and Smirnov's original definition of mean porosity (\cite[\S 2.2]{BeliaevSmirnov02}) is slightly more general. The results we obtain continue to hold with their definition; see \S\ref{subsec:other-notions} below.}

Lebesgue measure is trivially $(\alpha,1,\alpha^d)$-porous for any $\alpha$. Thus for mean porosity to be a meaningful concept, we need $\e< \alpha^d$.  We will informally say that $\mu$ is \textbf{weakly mean porous} if it is mean $(\alpha,\eta,\e)$-porous for some $\e \ll \alpha^d$.

From a geometric perspective, mean porous measures are perhaps more natural; at the very least, letting $\e\rightarrow 0$ results in cleaner statements and proofs. On the other hand, as pointed out in \cite[Section 6]{BeliaevSmirnov02}, many measures of dynamical and analytical origin are weakly porous, but not porous (or mean porous). In this article we obtain bounds on the packing dimension of weakly mean porous and, by extension, also of mean porous and weakly porous measures. We first state our main result in the mean porous case, and discuss its significance. Afterward we state our more general result concerning weakly mean porous measures.

\subsection{Dimension bound for mean-porous measures}

The following is our main theorem regarding the dimension of mean porous measures:

\begin{theorem} \label{thm:main-result}
Suppose $\mu$ is a mean $(\alpha,\eta)$-porous Borel probability measure on $\R^d$. Then
\[
\dimp\mu\le d - c_d\cdot \eta\cdot \alpha^d,
\]
where
\[
c_d = \frac{2 } {5\log 2\cdot 2^{4d}\cdot d^{d/2}},
\]
and $\dimp(\mu)$ is the packing dimension of $\mu$. (See \S\ref{subsec:dimension} for its definition.)
\end{theorem}

We make some remarks regarding this statement.

\begin{remark}
Other than for the value of the constant $c_d$, the estimate given in the above theorem is sharp. In other words, there exists a constant $c'_d>0$ such that for any $\alpha\in (0,1/2)$ and $\eta\in (0,1]$ one can construct a measure $\mu$ which is mean $(\alpha,\eta)$-porous and satisfies
\[
\dim_P(\mu) = \dim_H(\mu) \ge d -c'_d\cdot \eta\cdot \alpha^d.
\]
(Here $\dim_H(\mu)$ denotes the Hausdorff dimension of a measure; see \S\ref{subsec:dimension}.) This will be clear from the proof of the theorem. Alternatively, one can modify the example in \cite{KoskelaRohde97} which shows that \eqref{eq:dimension-mean-porous-sets} is sharp in the case of sets.
\end{remark}

\begin{remark}
For porous measures, a similar result was obtained independently by K\"{a}enm\"{a}ki, Rajala and Suomala \cite[Theorem 5.6]{KRS09}. Their proof works in regular metric spaces, under some assumptions on the measure, but does not extend to mean porosity.
\end{remark}

\begin{remark}
For $\alpha$ close to $1/2$ (which is the maximum possible value, see \cite[Section 2]{EJJ00}), a much better bound for the packing dimension of mean porous measures was obtained in \cite[Theorem 3.1]{BJJKRSS09}. However, to the best of our knowledge, for small $\alpha$ and any $\eta\in (0,1)$, it was not even known whether a mean $(\alpha,\eta)$-porous measure always has dimension strictly smaller than the dimension of the ambient space.
\end{remark}

\begin{remark}
Theorem \ref{thm:main-result} is stated in \cite[Theorem 1]{BeliaevSmirnov02} (for a slightly larger class of mean porous measures; see \S\ref{subsec:other-notions}). However, as pointed out in \cite{BJJKRSS09}, the proof given in \cite{BeliaevSmirnov02} is not correct. The problem lies in the claim that mean porous measures can be approximated by mean porous sets (see \cite[Proposition 1]{BeliaevSmirnov02}); in \cite{BJJKRSS09} it is shown that this is not the case: there are mean porous measures that assign zero mass to all mean porous sets. Therefore, it does not seem possible to prove Theorem \ref{thm:main-result} by reducing it to the results of Koskela and Rohde \cite{KoskelaRohde97} for sets.
\end{remark}

\begin{remark}
On the other hand, the estimate \eqref{eq:dimension-mean-porous-sets}  for mean porous sets does follow from our results. This can be seen from the following two facts: \begin{itemize}
\item[(a)] Any Borel set of packing dimension $\alpha$ supports a Borel probability measure of packing dimension at least $\alpha-\delta$ for any $\delta>0$ (see \cite{Cutler95}),
\item[(b)] If the support of a measure $\mu$ is mean $(\alpha,\eta)$-porous, then so is the measure $\mu$. This is obvious from the definitions.
\end{itemize}
\end{remark}

\subsection{Dimension bounds for weakly porous measures}

For many natural fractal measures, such as measures arising from hyperbolic dynamical systems, mean porosity is too strong of a condition. However, such measures are often mean $(\alpha,\eta,\e)$-porous for a suitable small, but positive, $\e$. See the discussion in \cite[Section 6]{BeliaevSmirnov02}. The next theorem shows that if $\e$ is smaller than some explicit multiple of $\alpha^d$, then the packing dimension of any mean $(\alpha,\eta,\e)$-porous measure is strictly smaller than the ambient dimension $d$. For sufficiently small $\e$, we recover the bound in Theorem \ref{thm:main-result}.

\begin{theorem} \label{thm:main-result-general}
Fix an ambient dimension $d$ and  $\alpha > 0$. There exists a continuous function $t_{d,\alpha}(\e)$ with the following properties:
\renewcommand{\theenumi}{\roman{enumi}}
\begin{enumerate}
\item \label{enum:positive-function} $t_{d,\alpha}(\e) > 0$ whenever $\e < 2^{-2d}d^{-d/2} \alpha^d$.
\item \label{enum:lim-as-epsilon-0} If $\e$ is small enough (depending on $d,\alpha$), then $t_{d,\alpha}(\e) > c_d\cdot \alpha^d$,
where $c_d$ is the constant in Theorem \ref{thm:main-result}. In particular, this holds for $\e=0$.
\item \label{enum:dimension-bound} If $\mu$ is mean $(\alpha,\eta,\e)$-porous, then
\[
\dim_P(\mu) \le d- \eta \cdot t_{d,\alpha}(\e).
\]
\end{enumerate}
\end{theorem}

Several remarks are in order.

\begin{remark} It will emerge from the proof that the function $t_{d,\alpha}$ can be determined explicitly. Indeed, we will see that
\[
t_{d,\alpha}(\e) =  2^{-d}t_{d,k}(\e),
\]
where $t_{d,k}$ is the function defined in Theorem \ref{thm:main-result-dyadic}, and $k=k(\alpha)$ is given by \eqref{eq:k-of-alpha-expl}.
\end{remark}

\begin{remark}
Note that Theorem \ref{thm:main-result} is an immediate corollary of parts \eqref{enum:lim-as-epsilon-0} and \eqref{enum:dimension-bound} in Theorem \ref{thm:main-result-general}.
\end{remark}

Even though Theorem \ref{thm:main-result-general} is a deterministic geometric result, our proof uses ideas from probability and dynamics. A first step is to convert the problem into a dyadic version which is more amenable to analysis. This is standard in geometric measure theory; here we obtain the dyadic analogue by means of a random translation (purely geometric arguments are often more complicated).

The key idea of the proof is to use trees and martingales to express the local dimension of a measure $\mu$ in terms of averages of entropies and Lyapunov exponents (suitably defined) as one zooms in towards a $\mu$-typical point. This idea was introduced (in a different context) in \cite{HochmanShmerkin09} to study the dimension of projected measures. A crucial difference with \cite{HochmanShmerkin09} is that here we need to consider dyadic partitions consisting of cubes of many different sizes (See Section \ref{sec:notation}). While dealing with uniform partitions is enough to show that the packing dimension of mean porous measures is strictly smaller than the dimension of the ambient space, in order to get the sharp asymptotics more general dyadic trees appear to be needed.

This semi-local approach has two crucial advantages. Firstly, porosity at a given scale results in an entropy drop at that scale, so that there is a natural link between porosity and dimension (via entropy averages). Secondly, since we average over all scales, this method is naturally suited to handle mean porosity, rather than just porosity.

\subsection{Organization of the paper} The paper is organized as follows. In Section \ref{sec:notation} we introduce our general setting and review some concepts of dimension of a measure. Section \ref{sec:martingales} describes a general method to estimate dimension which will be key in our later proofs. In Section \ref{sec:dyadic-porosity} we define dyadic analogues of mean porosity, state Theorem \ref{thm:main-result-dyadic}, which is an explicit dyadic version of our estimates, and deduce Theorem \ref{thm:main-result-general}. The proof of Theorem \ref{thm:main-result-dyadic} is given in Section \ref{sec:proof-of-theorem}. We conclude with examples, generalizations and open questions in Section \ref{sec:conclusion}.

\section{Notation and preliminaries} \label{sec:notation}

\subsection{Cube partitions and trees}

We fix an ambient dimension $d$ for the rest of the article. By a cube we always mean a half-open cube in $\R^d$ with sides parallel to the axes. The side length of a cube $Q$ will be denoted $\ell(Q)$.

Our main objects of interest will be trees of cubes formed by repeated subdivision. A \textbf{cube partition} of a cube $Q$ is a pairwise disjoint collection of cubes $\mathcal{R}=\{R_0,\ldots, R_{N-1}\}$ with $Q = \bigcup_{i\in [N]} R_i$. The cube partition $\mathcal{R}$ is \textbf{$\delta$-regular} if
\[
\delta \le \frac{\ell(R_i)}{\ell(Q)} \le 1-\delta\quad \text{for all } i\in [N].
\]
Suppose that to each cube $Q$ is associated a cube partition
\[
\mathcal{R}(Q) = \{R^Q_0,\ldots, R_{N_Q-1}^Q\}.
\]
Starting with the unit cube $[0,1)^d$, we can then form a tree $\RR^*$ by repeated subdivision of $Q$ into $\mathcal{R}(Q)$. (Of course, the partition needs not be defined for all cubes, but just for cubes which appear in the tree.) The tree is called $\delta$-regular if all partitions involved are $\delta$-regular. In this case, the degree of all vertices is uniformly bounded. The simplest such trees are the ones generated by subdivision into $b$-adic cubes for some $b\ge 2$, but we will need the more general construction.

Abusing notation, we will also denote the set of vertices of the tree by $\RR^*$. The collection of cubes at tree distance $n$ from the root $[0,1)^d$ will be denoted $\RR_n$. Furthermore, given $x\in [0,1)^d$, we will let $\RR_n(x)$ be the only cube in $\RR_n$ containing $x$.

A trivial but important fact is that, if $\sigma(\RR_n)$ is the $\sigma$-algebra generated by the elements of $\RR_n$, then the collection $\{ \sigma(\RR_n)\}$ is a filtration of $\sigma$-algebras which generates the Borel $\sigma$-algebra of $[0,1)^d$.

\subsection{Measures, entropy and Lyapunov exponents}

In this article it is understood that, absent explicit mention, all measures are Borel probability measures on $[0,1)^d$. Let $\RR^*$ be a tree as in the previous section. A measure $\mu$ induces a discrete probability measure $\mu^Q$ at each vertex $Q\in\RR^*$ such that $\mu(Q)>0$, supported on the offspring set $\RR(Q)$. Namely, if $R\in\RR(Q)$, then
\[
\mu^Q(R)  = \frac{\mu(R)}{\mu(Q)}.
\]
Conversely, any collection $\{\mu^Q\}$ of probability measures supported on $\RR(Q)$ gives rise to a Borel probability measure $\mu$. Note that if $Q_n\in\RR_n$ and $[0,1)^d=Q_0,\ldots, Q_{n-1}$ is the lineage of $Q_n$, then
\begin{equation} \label{eq:meas-cube-as-product-of-cond-meas}
\mu(Q_{n}) = \prod_{i=0}^{n-1} \mu^{Q_i}(Q_{i+1}).
\end{equation}
Suppose $Q\in \RR^*$ and $f$ is a function defined on the offspring set $\RR(Q)$. We write $\EE_{\mu^Q}$ for the expectation of $f$; explicitly,
\[
\EE_{\mu^Q}(f) = \sum_{R\in\RR(Q)} f(R)\, \mu^Q(R) .
\]

Given $Q\in\RR^*$, we define the \textbf{entropy} $H(Q)$ to be the entropy of the measure $\mu^Q$:
\[
H(Q)  = \EE_{\mu^Q}(-\log(\mu^Q(\cdot))) = \sum_{R\in\RR(Q)} -\log(\mu^Q(R)) \mu^Q(R).
\]
(Here we follow the usual convention $0\log 0=0$.) This quantity measures how ``spread out'' $\mu^Q$ is. However, because squares in $\RR(Q)$ may have different sizes, entropy alone is not a good characterization of the geometric ``size'' of $\mu^Q$. To this end we introduce the \textbf{Lyapunov exponent} $\lambda(Q)$. First, for $R\in \RR(Q)$, write
\[
\ell_Q(R) = \frac{\ell(Q)}{\ell(R)}.
\]
for the (inverse) relative length of the offspring cube. Then set
\[
\lambda(Q)  = \EE_{\mu^Q}(\log\ell_Q(\cdot)) = \sum_{R\in\RR(Q)}  \log(\ell_Q(R)) \,\mu(R).
\]

The quotient $H(Q)/\lambda(Q)$ is a discrete notion of dimension for $\mu^Q$. It is always bounded between $0$ and $d$; it is $0$ only when all mass is concentrated on one cube $R\in\RR(Q)$, and it is $d$ only when the mass of each $R$ is its (relative) Lebesgue measure.

\subsection{Dimensions of a measure} \label{subsec:dimension}

Recall that the \textbf{upper local dimension} of a measure $\mu$ at a point $x\in [0,1)^d$, denoted $\overline{\dim}(\mu,x)$, is defined as
\[
\overline{\dim}(\mu,x) = \limsup_{r\to 0} \frac{\log \mu(B(x,r))}{\log r},
\]
where $B(x,r)$ denotes the closed ball of radius $r$ centred at $x$. The (upper) \textbf{packing dimension} $\dim_P(\mu)$ is the essential supremum of the upper local dimensions:
\begin{equation} \label{eq:def-packing-dim}
\dim_P(\mu) = \mu\text{-}\esssup \overline{\dim}(\mu,x).
\end{equation}
Alternatively, $\dim_P(\mu) = \sup\{ \dim_P(A): \mu(A)>0\}$, where $\dim_P(A)$ is the packing dimension of the set $A$ (see \cite{Cutler95}).

For completeness, let us mention that the \textbf{Hausdorff dimension} of the measure $\mu$, denoted $\dim_H(\mu)$, is the essential supremum of the \emph{lower} local dimensions (defined in the obvious way). It coincides with $\inf\{\dim_H(A):\mu(A)>0\}$, where $\dim_H(A)$ denotes the Hausdorff dimension of the set $A$. Clearly, $\dim_H(\mu) \le \dim_P(\mu)$ in general, and strict inequality is possible. In particular, any upper bounds we prove for $\dim_P(\mu)$ are also upper bounds for $\dim_H(\mu)$.

We will be interested in obtaining upper bounds for the packing dimension, and thus we need upper bounds for the local dimension. The next lemma shows that it is enough to consider cubes in $\RR^*$ instead of balls.

\begin{lemma} \label{lem:dyadic-local-dim}
Suppose $\RR^*$ is $\delta$-regular for some $\delta>0$. Then
\[
\dim_P(\mu)= \mu\text{-}\esssup \limsup_{n\rightarrow\infty} \frac{ -\log\mu(\RR_n(x)) }{ \log\ell(\RR_n(x)) }.
\]
\end{lemma}
\begin{proof}
This follows from \cite[Theorem B.1]{KRS09}. Note that although $\RR_n$ are not $\delta_n$-partitions (in the terminology of \cite{KRS09}), a simple stopping time argument together with $\delta$-regularity can be used to construct $(1-\delta)^n$-partitions $\widetilde{\RR}_n$, such that any element of $\widetilde{\RR}_n$ is in $\RR_k$ for some $k\in [n]$. We can then apply \cite[Theorem B.1]{KRS09} to these partitions and observe that the lemma follows immediately.
\end{proof}

\begin{remark}
The upper bound in the lemma, which is all we need for the proofs of our main result, is elementary: it follows from $\delta$-regularity and the fact that $\RR_k(x)\subset B(x,\sqrt{d}\,\ell(\RR_k(x)))$.
\end{remark}

\section{Martingales and estimation of local dimension} \label{sec:martingales}

In this short section we prove a theorem that will be key in our dimension estimates. It is a variant of \cite[Lemma 4.2]{HochmanShmerkin09} (where only $b$-adic trees were considered; in particular, Lyapunov exponents do not arise in that context). Although the proof is a direct application of the law of large numbers for martingale differences, the theorem can be effective in a wide variety of situations.

\begin{theorem} \label{thm:local-dim-entropy-Lyapunov}
Suppose $\RR^*$ is $\delta$-regular for some $\delta>0$. Then for $\mu$-almost every $x$,
\begin{align}
\lim_{n\rightarrow\infty} \frac{1}{n} \left(-\log\mu(\RR_n(x)) -\sum_{i=0}^{n-1} H(\RR_i(x))\right) &= 0, \label{LLN-entropy}\\
\lim_{n\rightarrow\infty} \frac{1}{n} \left(\log\ell(\RR_n(x)) - \sum_{i=0}^{n-1} \lambda(\RR_{i}(x))  \right) &= 0. \label{LLN-length}
\end{align}
\end{theorem}
\begin{proof}
Define a sequence of functions $\{I_n\}_{n=1}^\infty$ on $[0,1)^d$ by
\[
I_n(x) = -\log \mu^{\RR_{n-1}(x)}(\RR_n(x)).
\]
It is easy to check that $I_n$ is $\RR_n$-measurable and uniformly bounded in $L^2(\mu)$. Moreover, $\EE(I_n|\RR_{n-1})=H(\RR_{n-1}(x))$. Thus
\[
\{ I_n(x) - H(\RR_{n-1}(x)) \}_{n=1}^\infty
\]
is a sequence of uniformly $L^2$-bounded martingale differences, whence by the law of large numbers for martingale differences (see \cite[Theorem 3 in Section 9 of Chapter 7]{Feller71}),
\[
\lim_{n\rightarrow\infty} \frac{1}{n} \sum_{i=1}^n \left( I_i(x) -H(\RR_{i-1}(x))\right) = 0 \quad \mu\textrm{-a.e.}
\]
But $\sum_{i=1}^n I_i(x) = -\log\mu(\RR_n(x))$ by \eqref{eq:meas-cube-as-product-of-cond-meas}, so we obtain \eqref{LLN-entropy}.

A similar argument, applied to the sequence of functions
\[
L_n(x) = \log\ell_{\RR_{n-1}(x)}(\RR_n(x)),
\]
yields \eqref{LLN-length}.
\end{proof}

As a consequence, we obtain the following corollary, which will be the key for estimating the packing dimension of porous measures.

\begin{corollary} \label{cor:bound-packing-dim-entropy-Lyapunov}
If $\RR^*$ is $\delta$-regular for some $\delta>0$, then
\[
\dim_P\mu = \mu\text{-}\esssup \limsup_{n\rightarrow\infty} \frac{\sum_{i=0}^{n-1} H(\RR_i(x))}{\log\ell(\RR_n(x))} .
\]
\end{corollary}
\begin{proof}
Since $\ell(\RR_n(x)) \ge \delta^n$, this is immediate from Lemma \ref{lem:dyadic-local-dim} and Theorem \ref{thm:local-dim-entropy-Lyapunov}.
\end{proof}

\begin{remark}
By the other part of Theorem \ref{thm:local-dim-entropy-Lyapunov}, the corollary still holds if the denominator is replaced by $\frac1n\sum_{i\in [n]} \lambda(\RR_n(x))$. However, we will require the less symmetric version above.
\end{remark}

\section{From porosity to dyadic porosity} \label{sec:dyadic-porosity}

The definition of mean porosity is in terms of Euclidean balls. In order to apply the machinery developed in the previous sections, we need an analogue involving dyadic cubes. We will denote by $\DD^*$ the tree generated by partitioning each dyadic cube into the dyadic sub-cubes of the next level. If $Q\in \DD_n$ and $k\ge 1$, we will write
\[
\DD_k(Q) = \{ R: \in \DD_{n+k}: R\subset Q\}.
\]
Given a measure $\mu$, $x\in [0,1)^d$, $n\ge 1$ and $0\le \e<1$, we let
\[
\por_2(\mu,x,n,\e) = \min\left\{ k: \min_{R\in \DD_k(\DD_n(x))}  \frac{\mu(R)}{\mu(\DD_n(x))}\le \e \right\}.
\]
(We use the convention $\min\emptyset = \infty$.) Given $k\ge 1$, $0<\eta\le 1$ and $0\le \e<1$, we say that a measure $\mu$ is \textbf{dyadic mean $(k,\eta,\e)$-porous} if
\begin{equation} \label{eq:def-dyadic-mean-lower-porosity}
\liminf_{n\rightarrow\infty} \frac{1}{n} |\{ i\in [n] : \por_2(\mu,x,i,\e)\le k \}| \ge \eta
\end{equation}
for $\mu$-almost every $x$. Finally, we say that $\mu$ is \textbf{dyadic mean $(k,\eta)$-porous} if it is dyadic mean $(k,\eta)$-porous for all $\e>0$. A similar, but more restricted, notion of dyadic porosity was introduced by Beliaev and Smirnov in \cite[\S 6.2]{BeliaevSmirnov02}.

The following is the main technical result of this paper. As we will see, it implies Theorem \ref{thm:main-result-general}, and also yields a considerably stronger version of \cite[Theorem 4]{BeliaevSmirnov02}.

\begin{theorem} \label{thm:main-result-dyadic}
Fix an ambient dimension $d$ and $k\ge 1$. Given $\e\in [0,2^{-kd}]$, Let $s(\e) = s_{d,k}(\e)$ be the largest real solution $s$ to the equation
\begin{equation} \label{eq:def-s}
(1-\e)\log\left(\frac{(2^d-1)(\sum_{i=1}^k 2^{-si})}{1-\e}\right) + \e\log(1/\e) = s\e\log(2^k).
\end{equation}
Further, let $t(\e) = t_{d,k}(\e) := d - s_{d,k}(\e)$. Then
\begin{equation} \label{eq:dimension-bound}
\dimp\mu \le d - \eta \cdot t_{d,k}(\e),
\end{equation}
for any dyadic mean $(k,\eta,\e)$-porous measure $\mu$ on $[0,1)^d$.
Moreover,
\renewcommand{\theenumi}{\roman{enumi}}
\begin{enumerate}
\item \label{enum:t-is-positive} $t_{d,k}(\e) > 0$ whenever $\e < 2^{-kd}$.
\item \label{enum:t-of-0-lower-bound} If $\e$ is small enough (depending on $d,k$),
\[
t_{d,k}(\e) > \frac{2}{5\log 2}\cdot 2^{-kd}.
\]
\end{enumerate}
\end{theorem}

Although the equation that defines $s_{d,k}$ looks fairly complicated, it is straightforward to compute it numerically. See Figure \ref{fig-t} for an example.

 \begin{figure}
    \centering
    \includegraphics[width=0.8\textwidth]{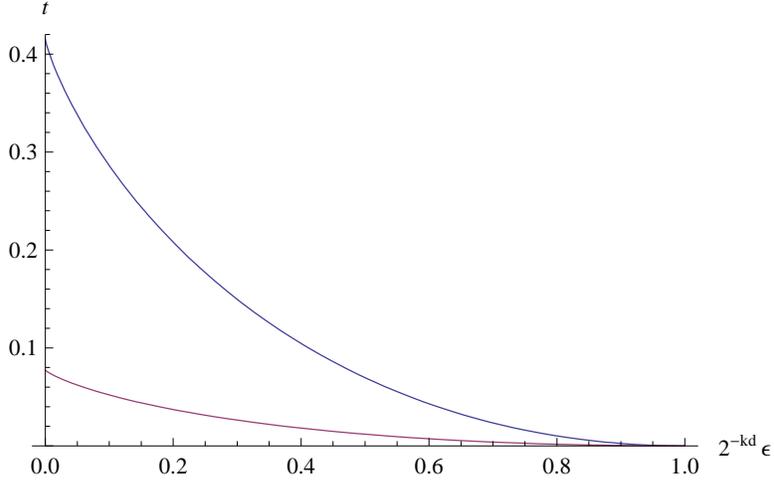}
    \caption{The ``dimension drop'' function $t_{d,k}(2^{-kd}\e)$ for $d=2$ and $k=1$ and $2$. }\label{fig-t}
  \end{figure}

\begin{remark}
We comment on the relationship of Theorem \ref{thm:main-result-dyadic} with \cite[Theorem 4]{BeliaevSmirnov02}. Roughly speaking, Beliaev and Smirnov consider an arbitrary base $b$ (which they denote by $k$), rather than just the dyadic base, but do not have a parameter equivalent to our $k$ (translating their setting into ours, their result corresponds to the case $k=1$). We remark that it is straightforward to give analogues of Theorem \ref{thm:main-result-dyadic} for an arbitrary base, and when $k=1$ these would allow us to recover the estimates in \cite[Theorem 4]{BeliaevSmirnov02}. However, in order to get the sharp estimates for dimension as the porosity $\alpha\to 0$, it is crucial to let $k\to\infty$ (while keeping the base fixed).

Also, Theorem \ref{thm:main-result-dyadic} holds for packing dimension and weak \emph{mean} porosity, while \cite[Theorem 4]{BeliaevSmirnov02} gives estimates on the Hausdorff dimension of weakly porous measures only.
\end{remark}

The next lemma shows that given a mean porous measure, one can find an appropriate dyadic frame such that the measure is also dyadic mean porous, at the cost of losing a constant factor in $\eta$ and $\alpha$.

\begin{lemma} \label{lem:porosity-to-dyadic-porosity}
Let $\mu$ be a mean $(\alpha,\eta,\e)$-porous measure, and fix $r\in (0,1/2)$.  Then for almost every $t\in [0,1/2)^d$, the measure $r \mu+t$ is $(k,(1-2r)^d\eta,\e)$-porous, where
\begin{equation} \label{eq:k-of-alpha}
k = \lceil|\log_2(\alpha r/\sqrt{d})|\rceil.
\end{equation}
\end{lemma}
\begin{proof}
Let $r\in (0,1/2)$ be fixed. Choose $t$ at random  uniformly in $[0,1/2)^d$, and consider the random measure $\widetilde{\mu}=r\mu+t$. Fix a point $x\in\supp\mu$ such that
\begin{equation} \label{eq:porous-point}
\liminf_{n\rightarrow \infty} \frac{1}{n}|\{i\in [n]:\por(\mu,x,2^{-i},\e)\ge \alpha\}| \ge \eta.
\end{equation}
Write $\widetilde{x}=r x+t$. The \textbf{relative position} of a point $y$ inside a dyadic cube $R$, denoted as $\text{pos}(y,R)$, is defined to be $T(y)$, where $T$ is the natural homothety mapping $R$ onto $[0,1)^d$. Note that $\widetilde{x}\bmod 1/2$ is a random variable whose distribution is Lebesgue measure on $[0,1/2)^d$, and therefore the relative positions $\text{pos}(\widetilde{x},\DD_n(\widetilde{x}))$, $n\ge 2$, form a sequence of i.i.d. uniformly distributed random variables.

Fix $r\in (0,1/2)$, and let $U_r$ be the set of points in $[0,1)^d$ at distance at least $r$ from the boundary $\partial [0,1)^d$. Since $U_r$ contains a cube of side $1-2r$, its Lebesgue measure is at least $(1-2r)^d$. Call $i$ a good scale if
\[
\por(\widetilde{\mu},\widetilde{x},r 2^{-i},\e)\ge \alpha \textrm{ and } \text{pos}(\widetilde{x},\DD_i(\widetilde{x}))\in U_r.
\]
By \eqref{eq:porous-point} and the law of large numbers, almost surely
\begin{equation} \label{eq:proportion-of-porous-central-scales}
\liminf_{n\rightarrow\infty} \frac{1}{n} |\{ i\in [n] : i \text{ is a good scale }  \}| \ge (1-2r)^d\eta.
\end{equation}
Note that any ball of radius $\alpha r 2^{-i}$ contains a dyadic cube of side length $2^{-(i+k)}$, where $k$ is as in \eqref{eq:k-of-alpha}. Therefore, if $i$ is a good scale then $\DD_i(\widetilde{x})$ contains a dyadic cube $R\in \DD_k(\DD_i(\widetilde{x}))$ satisfying
\begin{align*}
\widetilde{\mu}(R) & \le \widetilde{\mu}(B(x,\alpha r 2^{-i})) \\
&\le \e\widetilde{\mu}(B(x,r 2^{-i}))  \le \e\widetilde{\mu}(\DD_i(\widetilde{x})).
\end{align*}
In light of \eqref{eq:proportion-of-porous-central-scales}, almost surely
\[
\liminf_{n\rightarrow\infty} \frac{1}{n} |\{ i\in [n] : \por_2(\wt{\mu},\wt{x},i,\e)\le k \}| \ge (1-2r)^d\eta.
\]
The foregoing analysis is for a fixed $x$ satisfying \eqref{eq:porous-point}. Now, since $\mu$ is mean $(\alpha,\eta,\e)$-porous, $\mu$-a.e. point satisfies \eqref{eq:porous-point}, and therefore we can apply Fubini to conclude that for almost every $t$, the measure $r\mu+t$ is dyadic mean $(k, (1-2r)^d\eta,\e)$-porous, as desired.
\end{proof}

\begin{corollary} \label{cor:porosity-to-dyadic-porosity}
Let $\mu$ be a mean $(\alpha,\eta)$-porous measure of bounded support. Then there is a homothetic image $\wt{\mu}$ of $\mu$ which is dyadic mean $(k,2^{-d}\eta)$-porous, where
\begin{equation} \label{eq:k-of-alpha-expl}
k = \left\lceil \log_2(4\sqrt{d}/\alpha)\right\rceil.
\end{equation}
\end{corollary}
\begin{proof}
Without loss of generality $\mu$ is supported on $[0,1/2)^d$. Then apply Lemma \ref{lem:porosity-to-dyadic-porosity} with $r=1/4$ and a sequence $\e_n\to 0$.
\end{proof}

\begin{proof}[Proof of Theorem \ref{thm:main-result-general}]
Set
\[
t_{d,\alpha}(\e) = 2^{-d} \cdot t_{d,k}(\e),
\]
where $k=k(\alpha)$ is defined in \eqref{eq:k-of-alpha-expl}. The theorem follows immediately from Theorem \ref{thm:main-result-dyadic} and Lemma \ref{lem:porosity-to-dyadic-porosity} applied with $r=1/4$. To verify that the value of the constant $c_d$ is correct, note that $2^{-kd}\ge  2^{-3d}\cdot d^{-d/2}\cdot\alpha^d$. For the bound on $\e$ given in the first part, note that $t_{d,\alpha}(\e)>0$ whenever $\e <2^{-dk}$, and use \eqref{eq:k-of-alpha-expl}.
\end{proof}

\section{Proof of Theorem \ref{thm:main-result-dyadic}} \label{sec:proof-of-theorem}

\subsection{Construction of the tree}

From now on we assume that $\mu$ is a dyadic mean $(k,\eta,\e)$-porous measure, with $\e \in [0, 2^{-kd}]$.

We will split dyadic cubes according to whether they are porous or not: we say that $Q\in\DD^*$ is \textbf{porous} if there is $R\in\DD_k(Q)$ such that \begin{equation} \label{eq:R-has-small-measure}
\mu(R) \le \e \mu(Q).
\end{equation}
The class of all porous cubes will be denoted $\mathcal{P}$. We also let $\mathcal{N}=\mathcal{D}^*\backslash\mathcal{P}$ denote the class of all dyadic cubes which are not porous.

Notice that, since $\mu$ is dyadic mean $(k,\eta,\e)$-porous,
\begin{equation} \label{eq:application-porosity}
\liminf_{n\rightarrow\infty} \frac{1}{n} |\{i\in [n]: \DD_i(x)\in\mathcal{P}  \}| \ge \eta\quad \text{for } \mu\text{-a.e. } x.
\end{equation}

 \begin{figure}
    \centering
    \includegraphics[width=0.7\textwidth]{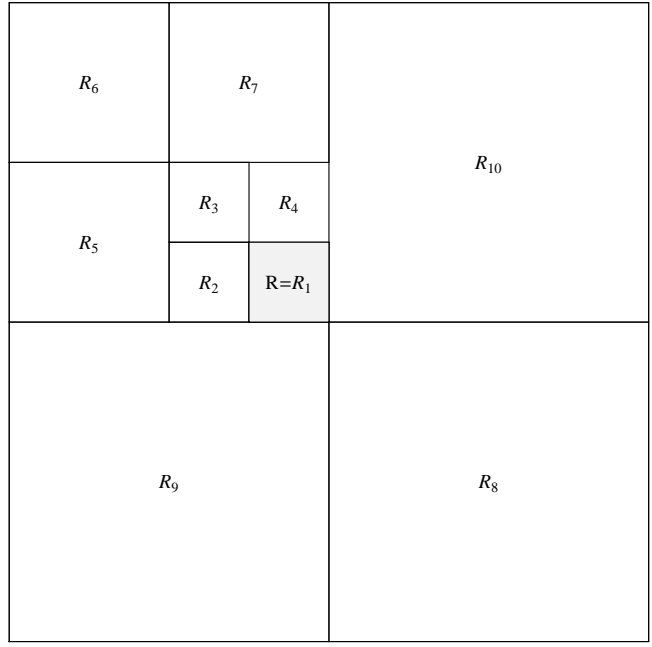}
    \caption{The construction of $\RR(Q)=\{R_1,\ldots, R_{10}\}$ when $Q\in\mathcal{P}$. The shaded square $R$ satisfies $\mu(R)\le \e \mu(Q)$. }\label{fig-tree}
  \end{figure}

We will now introduce the relevant partition operator $\RR$.
\begin{itemize}
\item If $Q\in\mathcal{N}$, then we let $\RR(Q):=\DD_1(Q)$ consist of the dyadic sub-cubes of $Q$ of first level.
\item If $Q\in\mathcal{P}$, then we define $\RR(Q)$ as follows. Let $R\in \DD_k(Q)$ be such that \eqref{eq:R-has-small-measure} holds. Then $\RR(Q)$ consists of $R$, and for each $1\le j\le k$, the $2^d-1$ dyadic cubes in $\DD_j(Q)$ which do not contain $R$. See Figure \ref{fig-tree} for an example when $d=2$ and $k=3$.
\end{itemize}

The tree $\RR^*$ is $2^{-k}$-regular, so our earlier results, and in particular Corollary \ref{cor:bound-packing-dim-entropy-Lyapunov}, apply.

\subsection{An estimate for $H(Q)/\lambda(Q)$}

Our goal is to estimate the packing dimension of $\mu$ by means of Corollary \ref{cor:bound-packing-dim-entropy-Lyapunov}. In order to do this we will need the following lemma, which describes the maximum possible value of $H(Q)/\lambda(Q)$ for $Q\in\mathcal{P}$.

\begin{lemma} \label{lem:calculation-constant}
The number $s(\e)=s_{d,k}(\e)$, defined in Theorem \ref{thm:main-result-general}, is the supremum of the possible values of $H(Q)/\lambda(Q)$ for $Q\in\mathcal{P}$ (over all possible measures $\mu$).
\end{lemma}
\begin{proof}
Let $N = (2^d-1)k + 1$, and let $( \alpha_i )_{i=1}^N$ be the vector
\[
\left(\underbrace{2^{-1},\ldots, 2^{-1}}_{2^d-1 \textrm{ times}},\ldots, \underbrace{2^{-(k-1)},\ldots,2^{-(k-1)}}_{2^d-1 \textrm{ times}},\underbrace{2^{-k},\ldots, 2^{-k}}_{2^d \textrm{ times}}\right).
\]
The supremum in question will be the supremum (indeed, the maximum) of the expression
\begin{equation} \label{eq:raw-expression-to-maximize}
\frac{\sum_{i=1}^{N} p_i\log(1/p_i)}{\sum_{i=1}^{N} p_i\log(1/\alpha_i)},
\end{equation}
subject to the constraints $\sum_{i=1}^{N} p_i = 1$, $p_i\ge 0$ for all $i$, and $p_N\le \e$. Write $L=2^d-1$. The first observation is that if (for a fixed $i$) we replace each of $p_{iL+1},\ldots, p_{iL+L}$ by the average $\frac1L \sum_{j=1}^{L} p_{iL+j}$, then the numerator in \eqref{eq:raw-expression-to-maximize} increases while the denominator stays constant and the constraints continue to hold. Hence, it is enough to show that $s(\e)$ is the maximum of
\begin{equation} \label{eq:reduced-expression-to-maximize}
\frac{ L\sum_{i=1}^k  \psi(q_i) +\psi(p) }{ L \sum_{i=1}^{k}  q_i \log(2^i) + p \log(2^k)},
\end{equation}
subject to $L \sum_{i=1}^k q_i=1-p$, $q_i\ge 0$ for all $i$ and $0\le p\le \e$. Here $\psi(t) = t\log(1/t)$ is the entropy function.

Fix $p\in [0,2^{-k}]$, and consider \eqref{eq:reduced-expression-to-maximize} as a function $g_p(q_1,\ldots, q_n)$ defined on the simplex
\[
\Delta_p = \left\{ (q_i) : L \sum_{i=1}^k q_i = 1-p \text{ and } q_i\ge 0 \text{ for all } i\right\}.
\]

\textbf{Claim}. Let $M$ be the global maximum of $g_p$ on $\Delta_p$. Then $M$ is attained at the point $q_i = A 2^{-M i}$, where $A$ is defined by the requirement that $(q_i)\in\Delta_p$.

To verify the claim, let
\[
f(q) = L\sum_{i=1}^k  \psi(q_i) +\psi(p) - M\left(L \sum_{i=1}^{k}  q_i \log(2^i) + p \log(2^k)\right),
\]
so that the global maximum of $f$ on $\Delta_p$ is $0$. Since $\frac{\partial f}{\partial q_i}|_{q_i=0^+}=+\infty$ (and partial derivatives are finite everywhere else), $f$ must attain its global maximum on $\text{interior}(\Delta_p)$. At the same time, a straightforward calculation using Lagrange multipliers shows that any local maximum of $f$ in $\text{interior}(\Delta_p)$ is attained at a point of the form $q_i =  A 2^{-M i}$. This yields the claim.

Now let
\[
\widetilde{\Delta}_\e = \left\{ (q_1,\ldots,q_k,p) : p\in [0,\e] \text{ and } (q_i)\in \Delta_p  \right\},
\]
and consider the function $G(q_1,\ldots,q_k,p):=g_p(q)$. A calculation similar to the proof of the claim shows that the only local maximum of $G$ on $\text{interior}(\widetilde{\Delta}_1)$ (i.e. without restrictions on $p$), is attained at a point with $p=2^{-k}$. It follows that if $\e<2^{-k}$, then the global maximum of $G$ on $\widetilde{\Delta}_\e$ has to be attained on the boundary. But it cannot be attained at a point with $q_i=0$ or $p=0$ (for the same reason as in the claim), so it has to be attained at a point with $p=\e$. Together with the claim, this shows that the maximum $M=M(\e)$ of $G$ on $\widetilde{\Delta}_\e$ is attained at the point $(A 2^{-M},\ldots, A 2^{-kM},\e)$.

Recalling the definitions and doing a little algebra, we find that $M(\e)$ satisfies the identity \eqref{eq:def-s}. Conversely, any other number $s$ verifying \eqref{eq:def-s}, satisfies $s = g_\e(A' 2^{-s},\ldots,A' 2^{-ks})$ (where $A'$ is such that the point is in $\Delta_\e$). As $M(\e)$ is the maximum of $g_\e$, we have $s\le M(\e)$. Thus $M(\e)$ is indeed the largest root of \eqref{eq:def-s}, which is what we wanted to prove.
\end{proof}

\subsection{Conclusion of the proof}
We start with some notation. Given $x\in [0,1)^d$ and $n\in\mathbb{N}$, let
\[
\mathcal{N}_n(x) = |\{i\in [n]: \RR_i(x)\in\mathcal{N} \}|.
\]
If $Q(\RR_n(x))\in \DD_M$, we write $M=M_n(x)$. In other words,
\[
M_n(x) = \log_2(\ell(\RR_n(x))).
\]
Set
\[
\eta_n(x) = 1-\frac{\mathcal{N}_n(x)}{M_n(x)}.
\]
We will make use of the following two lemmas in the proof of Theorem \ref{thm:main-result-dyadic}.

\begin{lemma} \label{lem:application-mean-porosity}
\[
\liminf_{n\rightarrow\infty} \eta_n(x) \ge   \eta \quad\text{ for } \mu\text{-a.e. } x.
\]
\end{lemma}
\begin{proof}
Clearly,
\[
\{ M_i(x) : \RR_i(x)\in\mathcal{N}, i\in [n]  \} \subset \{ j\in[M_n(x)]: \DD_j(x)\in\mathcal{N}  \},
\]
whence, by \eqref{eq:application-porosity},
\[
\limsup_{n\rightarrow\infty} \frac{1}{M_n(x)} |\{ M_i(x) : \RR_i(x)\in\mathcal{N}, i\in [n]  \}| \le 1-\eta,
\]
for $\mu$-almost every $x$. This implies the lemma.
\end{proof}

For the next lemma we need an additional bit of notation.  Write
\[
\overline{M}_n(x) = \log 2 (M_{n+1}(x)-M_n(x)) =  \log\left(\frac{\ell(\RR_n(x))}{\ell(\RR_{n+1}(x))}\right) .
\]

\begin{lemma} \label{lem:length-equal-lyapunov-exponents}
For $\mu$-almost every $x$,
\[
\lim_{n\rightarrow\infty}\frac{1}{n}\sum_{i\in [n]: \RR_i(x)\in\mathcal{P} } (\lambda(\RR_i(x))-\overline{M}_i(x))= 0.
\]
\end{lemma}
\begin{proof}
Since $\lambda(\RR_i(x)) = \overline{M}_i(x)=1$ for all $i$ such that $\RR_i(x)\in\mathcal{N}$, this follows from Theorem \ref{thm:local-dim-entropy-Lyapunov}.
\end{proof}

We are now ready to complete the proof of Theorem \ref{thm:main-result-dyadic}.

\begin{proof}[Proof of Theorem \ref{thm:main-result-dyadic}]
Fix a large $n$ and $x\in [0,1)^d$.
We want to estimate the quotient
\[
D_n(x) := \frac{\sum_{i=0}^{n-1} H(\RR_i(x))}{\sum_{i=0}^{n-1} \overline{M}_i(x)}.
\]
Indeed, by Corollary \ref{cor:bound-packing-dim-entropy-Lyapunov},
\begin{equation} \label{eq:application-bound-for-dimension}
\dimp(\mu) \le \mu\text{-}\esssup\limsup_{n\rightarrow\infty} D_n(x).
\end{equation}

If $\RR_i(x)\in \mathcal{N}$, then the situation is simple, as then $\overline{M}_{i}(x)=1$ and $H(Q)\le \log 2\cdot d$. Write
\begin{align*}
H_{\mathcal{P}}(n,x) &= \sum_{i\in [n]: \RR_i(x)\in\mathcal{P}} H(\RR_i(x)),\\
\overline{M}_{\mathcal{P}}(n,x) &= \sum_{i\in [n]: \RR_i(x)\in\mathcal{P}} \overline{M}_i(x).
\end{align*}
Then we have
\begin{align}
D_n(x) &\le \frac{ \log 2\cdot d\cdot \mathcal{N}_n(x) +  H_{\mathcal{P}}(n,x) }{ \log 2\cdot \mathcal{N}_n(x) + \overline{M}_{\mathcal{P}}(n,x)  }\nonumber\\
&=  d - \frac{d \cdot\overline{M}_{\mathcal{P}}(n,x)-H_{\mathcal{P}}(n,x)}{\log 2\cdot \mathcal{N}_n(x)+\overline{M}_{\mathcal{P}}(n,x)}.  \label{eq:partial-bound-1}
\end{align}
Note that, by the definition of $\eta_n(x)$,
\begin{align*}
\overline{M}_{\mathcal{P}}(n,x) &= \log 2\cdot(M_n(x) - \mathcal{N}_n(x))\\
 &= \log 2\cdot \eta_n(x) M_n(x)\\
 &= \frac{\eta_n(x)}{1-\eta_n(x)} (\log 2\cdot \mathcal{N}_n(x)).
\end{align*}
Thus, in light of \eqref{eq:partial-bound-1},
\[
D_n(x) \le d - \eta_n(x) \left(d-\frac{H_{\mathcal{P}}(n,x)}{\overline{M}_{\mathcal{P}}(n,x)}\right).
\]
Hence, by Lemmas \ref{lem:application-mean-porosity} and \ref{lem:length-equal-lyapunov-exponents},
\[
\limsup_{n\to\infty} D_n(x) \le d - \eta\left(d- \limsup_{n\rightarrow\infty} \frac{\sum_{i\in [n]: \RR_i(x)\in\mathcal{P}}H(\RR_i(x))} {\sum_{i\in [n]: \RR_i(x)\in\mathcal{P}}\lambda(\RR_i(x))}  \right),
\]
for $\mu$-almost every $x$. But, using Lemma \ref{lem:calculation-constant},
\begin{align*}
\frac{\sum_{i\in [n]: \RR_i(x)\in\mathcal{P}}H(\RR_i(x))} {\sum_{i\in [n]: \RR_i(x)\in\mathcal{P}}\lambda(\RR_i(x))}&\le \max_{i\in [n]: \RR_i(x)\in\mathcal{P}} \frac{H(\RR_i(x))}{\lambda(\RR_i(x))} \\
&\le s_{d,k}(\e),
\end{align*}
for all $x\in\supp(\mu)$. Recalling that $t_{d,k}(\e) = d- s_{d,k}(\e)$, we deduce that
\[
\limsup_{n\rightarrow\infty} D_n(x) \le d- \eta\cdot t_{d,k}(\e).
\]
for $\mu$-almost every $x$. In view of \eqref{eq:application-bound-for-dimension}, this concludes the proof of the dimension bound \eqref{eq:dimension-bound}.

To end the proof of the theorem, we observe that:
\begin{itemize}
\item \eqref{enum:t-is-positive} follows from the fact that $s_{d,k}$ is strictly increasing on $[0,2^{-dk}]$ (which can be seen, for example, from the proof of Lemma \ref{lem:calculation-constant}), and $s_{d,k}(2^{-dk}) = d$.
\item To see \eqref{enum:t-of-0-lower-bound}, by continuity of $s_{d,k}$ is enough to verify it at $\e=0$. See the proof of \cite[Theorem 4.1]{KaenmakiSuomala09} for the calculation in this case.
\end{itemize}
\end{proof}

\section{Examples, remarks and open questions}  \label{sec:conclusion}

\subsection{Other notions of mean porosity} \label{subsec:other-notions}

Our definition of mean porosity is somewhat artificial since it requires the existence of holes precisely at the dyadic scales (the class of mean $(\alpha,\eta$)-porous measures is not invariant under homotheties). In \cite[\S 2.2]{BeliaevSmirnov02}, a measure $\mu$ is defined to be mean $(\alpha,\eta)$-porous if for almost every $x$ there is $r(x)>0$, such that
\[
\liminf_{n\to\infty} \frac1n |\{i\in [n] : \por(\mu,x, r(x) 2^{-i},\e) \ge \alpha \}| \ge \eta \quad \text{for all } \e>0.
\]
It is clear that Theorem \ref{thm:main-result} continues to hold under this definition, at the cost of changing the constant $c_d$, since (assuming, as we may, $r(x)\in [0,1]$),
\[
\por(\mu,x, 2^{-i},\e)  \ge \frac{1}{2} \por(\mu,x, r(x) 2^{-i},\e).
\]
In fact, we claim that Theorem \ref{thm:main-result} holds for this larger class with the same constant. This is because, by compactness, for any $\widetilde{\alpha}<\alpha$ we can find a finite set $\{ r_i\}$ such that for $\mu$-almost every $x$, there is $i(x)$ (which can be chosen measurably) such that
\[
\liminf_{n\to\infty} \frac1n |\{i\in [n] : \por(\mu,x, r_{i(x)} 2^{-i},\e) \ge \widetilde{\alpha} \}| \ge \eta \quad \text{for all } \e>0.
\]
Applying Theorem \ref{thm:main-result} to each of the measures $r_j^{-1} \mu|_{\{x: i(x)=j\}}$ (which are now mean $(\widetilde{\alpha},\eta)$-porous according to our original definition), we find that $\dim_P(\mu) \le d - c_d \cdot\eta\cdot \widetilde{\alpha}^d$. As $\widetilde{\alpha}<\alpha$ was arbitrary, this verifies the claim.

Even though the definition of Beliaev and Smirnov makes the class of $(\alpha,\eta)$-porous measures invariant under homotheties, it is still tied to the base $2$ (if one replaces $2$ by $3$ in the definition, one gets a different class). We propose the following base-independent notion of mean porosity: let us say that $\mu$ is mean $(\alpha,\eta)$-porous if
\[
\liminf_{\rho\to 0} \frac{1}{\log(1/\rho)} \int_{\rho}^1 \mathbf{1}(\por(\mu,x,r,\e)>\alpha) \frac{dr}{r} \ge \eta,
\]
for $\mu$-a.e. $x$ and all $\e>0$. (Here $\mathbf{1}(A(r))$ is equal to $1$ if $A(r)$ holds, and to $0$ otherwise.) Again Theorem \ref{thm:main-result} continues to hold with this definition, at the price of changing the constant $c_d$, and even with the same constant, but this requires a finitary version of Theorem \ref{thm:local-dim-entropy-Lyapunov}.

\subsection{An example}

Deterministic measures arising from dynamics are often either absolutely continuous or weakly porous. Roughly speaking, this is because, if they are singular, then the scaling structure of the measure propagates macroscopic irregularities to all scales and points. However, random measures are often weakly mean porous (but not weakly porous). We discuss a simple model to illustrate this. Fix an ambient dimension $d$. Let $\mathbf{P}$ be a probability distribution on the simplex $\Delta$ of all probability measures on the set $\mathcal{D}_1$ of dyadic cubes of first level. We construct a random measure $\mu$ as follows: we first distribute a unit mass among all cubes in $\mathcal{D}_1$ by sampling $\mathbf{P}$. For each cube in $\mathcal{D}_1$ which was assigned positive mass, we further divide the mass along the cubes in $\mathcal{D}_1(Q)$ according to a new independent sampling from $\mathbf{P}$. We continue this process inductively. It is easy to see that if all cubes on $\mathcal{D}_1$ are given positive mass almost surely, then the support of $\mu$ is almost surely the unit cube. However, the measure $\mu$ is in general weakly mean porous:

\begin{prop}
Let $\mathbf{uniform}\in\Delta$ denote the probability distribution on $\mathcal{D}_1$ which assigns the same mass $2^{-d}$ to each cube. Suppose $\mathbf{P}$ is not concentrated on $\mathbf{uniform}$. Then there are $\eta\in (0,1)$ and $\e\in (0,2^{-d})$ (depending on $\mathbf{P}$) such that $\mu$ is dyadic mean $(1,\eta,\e)$-porous almost surely.

Moreover, if $\mathbf{uniform}$ is in the support of $\mathbf{P}$, then almost surely $\mu$ is not (dyadic) weakly porous. More precisely, in this case for every $k\in\mathbb{N}$ and $\delta\in (0,2^{-d})$, there is $\sigma=\sigma(k,\delta)\in (0,1)$ such that $\mu$ is \emph{not} dyadic mean $(k,\sigma,\delta)$-porous almost surely.
\end{prop}
\begin{proof}
This is an easy consequence of the definitions; details are left to the reader.
\end{proof}

\subsection{A converse to Theorem \ref{thm:main-result-dyadic}}

We have seen in Theorem \ref{thm:main-result-dyadic} that if a measure is (dyadic) weakly mean porous, then its packing dimension is smaller than the dimension of the ambient space. The converse implication also holds:

\begin{prop}
Let $\mu$ be a measure on $[0,1)^d$ such that $\dim_P(\mu)<d$. Then there are $\eta\in (0,1)$ and $\e\in (0,2^{-d})$ such that $\mu$ is dyadic mean $(1,\eta,\e)$-porous.
\end{prop}
\begin{proof}
Let $H_{\min}(\e)$ denote the smallest possible entropy of a probability vector $(p_1,\ldots,p_{2^d})$ subject to the constraint $p_i\ge\e$ for all $i$. Clearly,
\[
\lim_{\e\to 2^{-d}} H_{\min}(\e) = \log(2^d).
\]
Now suppose $\mu$ is not dyadic mean $(1,\eta,\e)$-porous. Then there exists a set $A$ of positive $\mu$-measure such that
\[
\liminf_{n\rightarrow\infty} \frac{1}{n} |\{ i\in [n] : \por_2(\mu,x,i,\e)\le 1 \}| < \eta
\]
for all $x\in A$. If $\por_2(\mu,x,i,\e)>1$, then $H(\DD_i(x))\ge H_{\min}(\e)$. Hence, using Corollary \ref{cor:bound-packing-dim-entropy-Lyapunov},
\begin{align*}
\dim_P(\mu) &\ge \inf_{x\in A}\limsup_{n\rightarrow\infty} \frac{\sum_{i=0}^{n-1}  H(\RR_i(x))}{\log 2\cdot n}
 \\ &\ge (1-\eta) \frac{H_{\min}(\e)}{\log 2}.
\end{align*}
Since
\[
 (1-\eta) \frac{H_{\min}(\e)}{\log 2} \rightarrow d \quad \textrm{ as }\quad \eta\to 0, \e\to 2^{-d},
\]
this yields the proposition.
\end{proof}

\subsection{Open questions}

We finish the paper with two open questions.

\begin{question}
In Theorem \ref{thm:main-result}, is it possible to take the constant $c_d$ independent of the ambient dimension $d$? (Note that in Theorem \ref{thm:main-result-dyadic} the same constant does work for all dimensions.) If not, what is the sharp rate of increase for $c_d$ as $d\rightarrow\infty$?
\end{question}

\begin{question}
Does Theorem \ref{thm:main-result} remain valid in Ahlfors-regular metric spaces? As in \cite{KRS09}, it may be necessary to make additional assumptions on the measure.
\end{question}

\bigskip

\noindent{\bf Acknowledgement}. I am grateful to A. K\"{a}enm\"{a}ki and V. Suomala for useful discussions regarding porosity and many helpful comments on earlier versions of the manuscript.

%\bibliographystyle{plain}
%\bibliography{porosities}

\begin{thebibliography}{10}

\bibitem{BJJKRSS09}
D.~Beliaev, E.~J\"{a}rvenp\"{a}\"{a}, M.~J\"{a}rvenp\"{a}\"{a},
  A.~K\"{a}enm\"{a}ki, T.~Rajala, S.~Smirnov, and V.~Suomala.
\newblock Packing dimension of mean porous measures.
\newblock {\em J. London Math. Soc.}, 80(2):514--530, 2009.

\bibitem{BeliaevSmirnov02}
D.~B. Beliaev and S.~K. Smirnov.
\newblock On dimension of porous measures.
\newblock {\em Math. Ann.}, 323(1):123--141, 2002.

\bibitem{Chousionis09}
Vasilis Chousionis.
\newblock Directed porosity on conformal iterated function systems and weak
  convergence of singular integrals.
\newblock {\em Ann. Acad. Sci. Fenn. Math.}, 34(1):215--232, 2009.

\bibitem{Cutler95}
Colleen~D. Cutler.
\newblock Strong and weak duality principles for fractal dimension in
  {E}uclidean space.
\newblock {\em Math. Proc. Cambridge Philos. Soc.}, 118(3):393--410, 1995.

\bibitem{EJJ00}
Jean-Pierre Eckmann, Esa J{\"a}rvenp{\"a}{\"a}, and Maarit
  J{\"a}rvenp{\"a}{\"a}.
\newblock Porosities and dimensions of measures.
\newblock {\em Nonlinearity}, 13(1):1--18, 2000.

\bibitem{Feller71}
William Feller.
\newblock {\em An introduction to probability theory and its applications.
  {V}ol. {II}.}
\newblock Second edition. John Wiley \& Sons Inc., New York, 1971.

\bibitem{HochmanShmerkin09}
Michael Hochman and Pablo Shmerkin.
\newblock Local entropy averages and projections of fractal measures.
\newblock {\em Preprint}, 2009.

\bibitem{JJKRRS09}
E.~J{\"a}rvenp{\"a}{\"a}, M.~J{\"a}rvenp{\"a}{\"a}, A.~K{\"a}enm{\"a}ki,
  T.~Rajala, S.~Rogovin, and V.~Suomala.
\newblock Packing dimension and ahlfors regularity of porous sets in metric
  spaces.
\newblock {\em Math. Z.}, To appear, 2009.

\bibitem{JJKS05}
E.~J{\"a}rvenp{\"a}{\"a}, M.~J{\"a}rvenp{\"a}{\"a}, A.~K{\"a}enm{\"a}ki, and
  V.~Suomala.
\newblock Asympotically sharp dimension estimates for {$k$}-porous sets.
\newblock {\em Math. Scand.}, 97(2):309--318, 2005.

\bibitem{JarvenpaaJarvenpaa02}
Esa J{\"a}rvenp{\"a}{\"a} and Maarit J{\"a}rvenp{\"a}{\"a}.
\newblock Porous measures on {$\Bbb R^n$}: local structure and dimensional
  properties.
\newblock {\em Proc. Amer. Math. Soc.}, 130(2):419--426 (electronic), 2002.

\bibitem{JJM02}
Esa J{\"a}rvenp{\"a}{\"a}, Maarit J{\"a}rvenp{\"a}{\"a}, and R.~Daniel Mauldin.
\newblock Deterministic and random aspects of porosities.
\newblock {\em Discrete Contin. Dyn. Syst.}, 8(1):121--136, 2002.

\bibitem{Kaenmaki07}
Antti K{\"a}enm{\"a}ki.
\newblock Porosity and regularity in metric measure spaces.
\newblock {\em Real Anal. Exchange}, (31st Summer Symposium
  Conference):245--250, 2007.

\bibitem{KRS09}
Antti K{\"a}enm{\"a}ki, Tapio Rajala, and Ville Suomala.
\newblock Local homogeneity and dimensions of measures in doubling metric
  spaces.
\newblock {\em Preprint}, 2009.

\bibitem{KaenmakiSuomala08}
Antti K{\"a}enm{\"a}ki and Ville Suomala.
\newblock Conical upper density theorems and porosity of measures.
\newblock {\em Adv. Math.}, 217(3):952--966, 2008.

\bibitem{KaenmakiSuomala09}
Antti K\"{a}enm\"{a}ki and Ville Suomala.
\newblock Nonsymmetric conical upper density and $k$-porosity.
\newblock {\em Trans. Amer. Math. Soc.}, To appear, 2009.

\bibitem{KoskelaRohde97}
Pekka Koskela and Steffen Rohde.
\newblock Hausdorff dimension and mean porosity.
\newblock {\em Math. Ann.}, 309(4):593--609, 1997.

\bibitem{MartioVuorinen87}
O.~Martio and M.~Vuorinen.
\newblock Whitney cubes, {$p$}-capacity, and {M}inkowski content.
\newblock {\em Exposition. Math.}, 5(1):17--40, 1987.

\bibitem{Mattila88}
Pertti Mattila.
\newblock Distribution of sets and measures along planes.
\newblock {\em J. London Math. Soc. (2)}, 38(1):125--132, 1988.

\bibitem{Mattila95}
Pertti Mattila.
\newblock {\em Geometry of sets and measures in {E}uclidean spaces}, volume~44
  of {\em Cambridge Studies in Advanced Mathematics}.
\newblock Cambridge University Press, Cambridge, 1995.
\newblock Fractals and rectifiability.

\bibitem{Nieminen06}
Tomi Nieminen.
\newblock Generalized mean porosity and dimension.
\newblock {\em Ann. Acad. Sci. Fenn. Math.}, 31(1):143--172, 2006.

\bibitem{Rajala09}
Tapio Rajala.
\newblock Large porosity and dimension of sets in metric spaces.
\newblock {\em Ann. Acad. Sci. Fenn. Math.}, 34(2):565--581, 2009.

\bibitem{Salli91}
Arto Salli.
\newblock On the {M}inkowski dimension of strongly porous fractal sets in
  {${\bf R}^n$}.
\newblock {\em Proc. London Math. Soc. (3)}, 62(2):353--372, 1991.

\bibitem{Urbanski01}
Mariusz Urba{\'n}ski.
\newblock Porosity in conformal infinite iterated function systems.
\newblock {\em J. Number Theory}, 88(2):283--312, 2001.

\end{thebibliography}

\end{document}